\documentclass[12pt,notitlepage,reqno]{amsart}

\usepackage[margin=3cm,footskip=1cm]{geometry}

\usepackage{amssymb} \usepackage{amsmath} \usepackage{amsfonts}
\usepackage{amsthm} \usepackage{mathtools}
\usepackage{esint}

\usepackage{mathabx}
\usepackage{enumerate} \usepackage[latin1]{inputenc}
\usepackage[shortlabels]{enumitem}

\usepackage{color}
\usepackage{graphicx} \usepackage{verbatim}

\newcommand{\R}{{\mathbb{R}}}

\renewcommand{\d}{{\rm d}}

\theoremstyle{plain}
\newtheorem{thm}{Theorem}[section]
\newtheorem{proposition}[thm]{Proposition}
\newtheorem{lemma}[thm]{Lemma} \newtheorem{corollary}[thm]{Corollary}
\theoremstyle{definition}

 \newtheorem{remark}[thm]{Remark}

\newtheorem*{remarks*}{Remarks}
\newtheorem*{remark*}{Remark}
\newtheorem*{defn*}{Definition}
\newtheorem{conjecture}[thm]{Conjecture}

\numberwithin{equation}{section}

\title{The Howland - Kato Commutator Problem, II}

\author{Ira Herbst}
\address{Department of Mathematics \\
  University of Virginia \\
  Charlottesville \\
  VA 22904\\ U.S.A.}
\email{iwh@virginia.edu}

\date{\today}

\begin{document}

\begin{abstract}
We continue to investigate the following problem: For what bounded measurable real $f$ and $g$ is the commutator $i[f(P),g(Q)]$ positive?   In this work we consider the situation where the commutator is a finite rank operator.
\end{abstract}

 \maketitle

\section{Introduction}
This is the second in a series of papers on the Howland-Kato commutator problem.  We refer the reader to the original paper of Kato [1] and the paper [3]. \\ 

We consider real, bounded, measurable $f$ and $g$.  If $i[f(P),g(Q)] \ge 0$, it follows that the commutator is trace class ([1],[3]). Thus it is natural to consider the case where the commutator is finite rank.  Kato considered the (non-zero) rank one case and showed that in this case $f$ and $g$ were essentially multiples of the hyperbolic tangent. Kato assumed that $f$ and $g$ were absolutely continuous with $L^1$ derivatives but this can be proved, see [3]. Specifically Kato showed $g(x) = c_1\tanh\hat r (x-t) + c_2, f(\xi) = c_1' \tanh \hat r' (\xi - t') + c_2'$ where the $c_j$ and $c_j'$ are real constants with $c_1c_1' > 0$ and $\hat r \hat r' = \pi/2$. (Kato derived a differential equation for $g$ which he then solved.)  If we define $K_r$ as the set of all bounded real functions, $f:\R \rightarrow \R$ with an analytic continuation to the strip $|\text{Im} z| < r$ satisfying the additional condition that $\text{Im}f(z) \text{Im} z \ge 0$ in this strip, then the functions that Kato found in the rank one case satisfy $sg\in K_r, sf\in K_{r'}$ where $\hat r = \pi/2 r, \hat r' = \pi/2r'$, and $s=\pm 1$.  Kato felt there was reason to believe that all pairs of $g$ and $f$ would be given by $g = s \tanh( \hat r \cdot)*d\mu + c$  and $f = s\tanh( \hat r'\cdot)* d\nu + c'$ where $\mu$ and $\nu$ are finite positive measures, $s= \pm 1$, $c$ and $c'$ are constants, and $rr'=\pi/2$.  We will call this the Kato conjecture.  In this paper we will consider real bounded measurable $f$ and $g$ such that

\begin{equation} \label{commutatorassumption}
i[f(P),g(Q)] = \sum_{j = 1}^N (\phi_j,\cdot)\phi_j
\end{equation}
where the $\phi_j$'s are linearly independent functions in $L^2(\mathbb{R})$.  From [3] we know that if the commutator is non-zero as well as non-negative then $f$ and $g$ are monotone.  Without loss of generality we will always assume that they are increasing.  In the next section we will show that $g \in K_r$ and $ f\in K_{r'}$ for some positive $r$ and $r'$. In addition we will prove some exponential estimates.  In his paper Kato showed that $g\in K_r$ if and only if $g(x) = \int \tanh \hat r(x-t) d\mu(t) + c$ where $\hat r = \pi/2r, \mu$ is a finite positive measure, and $c$ is a real constant.  We will make use of this representation.  In addition we will make use of the integral kernel, $K(x,y)$, of the commutator $i[f(P),g(Q)]$ given by 

\begin{equation} \label{kernel}
K(x,y) = \frac{1}{\sqrt{2\pi}} \frac{g(x) - g(y)}{x-y} \widehat{f'}(y-x) = \sum_j\phi_j(x)\overline{\phi_j(y)}
\end{equation}
  In the case at hand (finite rank, non-zero, non-negative commutator) $f'$ is in $L^1$, $g$ is $C^1$, and this expression is well-defined. In fact $f', g'$ and the $\phi_j$'s are in $\mathcal{S}(\mathbb{R})$ (see [3]). If we conjugate (\ref{commutatorassumption}) with the Fourier transform we obtain 

$$i[-g(-P),f(Q)] = \sum_{j = 1}^N (\widehat{\phi_j},\cdot)\widehat{\phi_j}$$
with integral kernel $\tilde{K}(\xi,\eta)$ given by

\begin{equation} \label{Fkernel}
\tilde{K}(\xi,\eta) = \frac{1}{\sqrt{2\pi}} \frac{f(\xi) - f(\eta)}{\xi-\eta} \widehat{g'}(\xi-\eta) = \sum_j\widehat{\phi_j}(\xi)\overline{\widehat{\phi_j}(\eta)}
\end{equation}

Thus if we prove something about $g$ the same proof gives an analogous statement for $f$ and vice versa.  

\section{Acknowledgement}

Useful conversations with Richard Froese are gratefully acknowledged.

\section{analyticity}
 
\begin{thm}\label{minimalrr'}
With $f$ and $g$ as above, assume without loss of generality that $f$ and $g$ are monotone increasing (see [3]).  Then there exist $r$ and $r'$, both positive so that $g\in K_r$ and $f\in K_{r'}$.  The functions $\phi_j$ are analytic in the strip $S_r: = \{z \in \mathbb{C}: |\text{Im} z| < r\}$.  They satisfy $\int|\phi_j(x+iy)|^2 e^{2s|x|}dx < \infty$ for $|y| < r$ and $s < r'$.  As a consequence $\int g'(x)e^{2s|x|} dx < \infty$ for $s < r'$ and $\int f'(\xi)e^{2s|\xi|}d\xi < \infty$ if $s < r$.  We have $rr'\le \pi/2$.
\begin{remark}
Note that $\int|\phi_j(x+iy)|^2e^{2sx}dx = \int |\hat \phi_j(\xi + is)|^2e^{-2y\xi}d\xi$.
\end{remark}
\end{thm}
If Kato's conjecture is true, what is required is that $rr' \ge \pi/2$. (Note that $rr' >\pi/2$ implies 
that there exist $r_1 \le r $ and $r_1' \le r'$ with $g\in K_{r_1}$ and $f\in K_{r'_1}$ such that $r_1r_1' = \pi/2$ .
This is a consequence of the fact that $K_a \subset K_b$ if $a > b$).    The result that $rr'\le \pi/2$ is therefore a bit surprising.  But evidently $rr'>\pi/2$ cannot occur in the finite rank case. \

We remark that our method of proof of the analyticity cannot predict the correct size of the strips.  The positivity of the commutator is not used in our proof of analyticity but rather only to prove the positivity of the imaginary part of $g$ (or $f$) in the upper half strip of analyticity.  If we just assume the finite rank property, then the proof of analyticity still works as given here but the following is an example where the commutator has finite rank but is neither positive nor negative and has $f$ and $g$ analytic in strips:  Take $g(x) = \tanh x$ and $f(\xi) = \tanh \alpha_1 \xi + \beta \tanh \alpha_2 \xi$  where $\alpha_1 = \pi/2$, $\alpha_2 = \pi$, and $\beta$ is small and positive.  After diagonalizing the rank two operator $i[\tanh (\alpha_2 P), g(Q)]$ the rank three commutator $i[f(P),g(Q)]$ can be written 
$$ i[f(P),g(Q)] = \pi^{-1}(\phi, \cdot)\phi + (\beta/2\pi)[||\phi_{+}||^{-2} \lambda_+(\phi_+, \cdot) \phi_+ + \lambda_{-} ||\phi_{-}||^{-2} (\phi_{-},\cdot)\phi_{-} ]$$
where $$\phi(x) = (\cosh x)^{-1}, \phi_{+} = \frac{\cosh(x/2)}{\cosh x}, \phi_{-} = \frac{\sinh(x/2)}{\cosh x},$$ $$ \lambda_{\pm} = ||\phi||^2 \pm ||e^{-x/2}\phi||^2.$$
One sees $(\phi, \phi_{-}) = 0$ and thus 
$$ (\phi_{-}, i[f(P),g(Q)]\phi_-) = (\beta/2\pi)||\phi_-||^4 \lambda_{-1} < 0.$$
where we used $\lambda_- = ||\phi||^2 - ||e^{-x/2}\phi||^2  = (\phi,\phi) - (\phi,1) < 0$ .  Thus $g$ is analytic in $S_{\pi/2}$ while $f$ is analytic in $S_{1/2}$.  The product is $(\pi/2)(1/2) = \pi/4$, half the required number, but of course the commutator is not positive.  
\vspace{.2cm}

Going back to Theorem \ref{minimalrr'} we do not know how large $r$ and $r'$ are.  But what we do know is that if $f$ and $g$ are analytic and bounded in larger strips than we have estimated, then their imaginary parts are positive in the corresponding upper half strips:
\begin{proposition}\label{largerrr'}
Suppose $f$ and $g$ are as in Theorem \ref{minimalrr'}.  If $\tilde r \ge r$ and $\tilde r' \ge r'$ and $g$ and $f$ are analytic in $S_{\tilde r}$ and $S_{\tilde r'}$ respectively and bounded in any smaller strips, then $g\in K_{\tilde r}$ and $f\in K_{\tilde r'}$.  In addition $\int g'(x)e^{2s|x|} dx < \infty$ for $s < \tilde r'$ and $\int f'(\xi)e^{2s|\xi|}d\xi < \infty$ if $s < \tilde r$. 
\end{proposition}
 
The fact that analyticity and boundedness of $g$ in a larger strip also implies $\text{Im}z\text{Im}g(z)\ge 0$ in this larger strip may also be special to the finite rank case.  It is not required by Kato's conjecture.  (Kato claims that a specific function given in his paper illustrates that one may have $f \in K_r$ and $f$ analytic and bounded in $S_{r_1}$ with $r_1 > r$ but $f \notin K_{r_1}$.  But I do not believe the example has this property.)

Let us define maximality:  If $g \in K_a$ and in no strictly smaller $K_b$ ($b>a$) then we say that $K_a$ is maximal (for $g$).  We will always be talking about non-constant functions.

\begin{conjecture}
Suppose $f \in K_a, g\in K_b$ with both $K_a$ and $K_b$ maximal for $f$ and $g$ respectively.  If  $i[f(P),g(Q)] \ge 0$, then $ab\ge \pi/2$. 
\end{conjecture}

This is implied by the Kato conjecture and if true would settle the case of finite rank positive commutators.  

\begin{proof} [Proof of Theorem \ref {minimalrr'}] 

Define the vector $v(x) = <\phi_1(x),...,\phi_N(x)>$ in $\mathbb{C}^N$.  Note that the set $\{v(y): y\in \mathbb R\}$ spans $\mathbb {C}^N$. Otherwise there is a non-zero vector $w \in \mathbb{C}^N$ orthogonal to all $v(x), x\in \mathbb{R}$ in which case $\sum \overline{w_j}\phi_j(x) = 0 $ for all $x \in \mathbb{R}$  which contradicts linear independence. 
Let $\{v(y_1),...,v(y_N)\}$ be a set of linearly independent vectors in $\mathbb{C}^N$.  And define (with $(\cdot,\cdot)$ the inner product in $\mathbb{C}^N$),

$$\gamma_j(x) = \frac{1}{\sqrt{2\pi}}\frac{g(x)-g(y_j)}{x-y_j}\hat{f'}(y_j-x) = (v(y_j),v(x)) = \sum_k\phi_k(x)\overline {\phi_k}(y_j)$$

The matrix $M_{kj} = \overline{\phi_k(y_j)}$ is invertible since $\sum_j \overline{M_{kj}}c_j = 0$ for a non-zero vector $c$ implies $\sum _j c_jv(y_j) = 0$ which is not possible since the $v(y_j)$ are linearly independent. Thus is is easy to go back and forth between $\gamma_j$ and $\phi_j = \sum_k\gamma_k M^{-1}_{kj}$.  We will need a disjoint set of N $y_j$'s, call them $y_j'$ such that the $v(y_j')$ are also linearly independent.  Such a set can be found since $v(y)$ is continuous.  Thus we have both $\gamma_j$ and $\tilde{\gamma_j}$ with the latter constructed using the $y_j'$. The technical point is to deal with the apparent singularity when $x=y_j$.   
We proceed by estimating derivatives.  According to [3], $f'$, $g'$, and the $\phi_j$'s are in $\mathcal{S}(\R)$. Consider 

$$\gamma_j^{(n)}(x)  = \frac{1}{\sqrt{2\pi}}\sum_{k=0}^n \binom{n}{k}\Big(\frac{g(x) - g(y_j)}{x-y_j}\Big)^{(k)} i^{(n-k)} \widehat {\xi^{n-k}f'}(y_j-x)$$

We have $(g(x) - g(y))/(x-y) = \int_0^1 g'(x + t(y-x)) dt$.  Thus 

$$(d^k/dx^k)\frac{g(x) - g(y)}{x-y} = \int_0^1 (1-t)^k g^{(k+1)}(x +t(y-x))dt.$$  
We integrate by parts to find 

$$\int_0^1 (1-t)^k g^{(k+1)}(x + t(y-x))dt = (y-x)^{-1}\int _0^1 (g^{(k)}(x+t(y-x))-g^{(k)}(x)) k(1-t)^{k-1}dt$$ unless $k=0$ in which case we are back to $(g(y) - g(x))/(y-x)$. (Note the above even makes some sense when $k = 0$ since as $k\to 0 , k(1-t)^{k-1}$ approaches a delta function at $1$.)   
Thus we have

$$|\big((g(y) - g(x))/(x-y)\big)^{(n)}| \le 2|y-x|^{-1} ||g^{(n)}||_\infty.$$
$$||\big((g(y) - g(x))/(x-y)\big)^{(n)}1_{|x-y| >b}||_2 \le (2\sqrt{2}) b^{-1/2} ||g^{(n)}||_\infty.$$
We have 

$$|\gamma_j^{(n)}(x)| \le \frac{1}{\sqrt{2\pi}}2|x-y_j|^{-1} \sum_{k\le n}\binom{n}{k} ||g^{(k)}||_{\infty}||\widehat{\xi^{n-k}f'}||_\infty $$

In computing $||\phi_k^{(n)}||_\infty$ we can use $\gamma_j$ or $\tilde \gamma_j$.  Thus we must bound the minimum of $\sum_j |x-y_j|^{-1}$ and $\sum_j |x-y'_j|^{-1}$.  If  $S = \{y_1,...y_N \}$ and  $S' = \{y'_1,...y'_N \}$, we have $|x-y_j| \ge d(S, x)$ thus $\sum_j |x-y_j|^{-1} \le N/d(S, x)$.  Hence the minimum of the two sums is $\le N \min \{d(S, x)^{-1}, d(S', x)^{-1}\} \le 2Nd(S,S')^{-1}$. Thus we have 

\begin{equation} \label{phibound}
||\phi_j^{(n)}||_\infty \le C \sum_{k\le n}\binom{n}{k} ||g^{(k)}||_{\infty}||\widehat{\xi^{n-k}f'}||_\infty
\end{equation}
\begin{align*}
    &||\gamma_j^{(n)}1_{d(S, x)>b}||_2 \le ||\gamma_j^{(n)}1_{|x-y_j| >b}||_2 \le \\ \nonumber
&(2\sqrt{2}) b^{-1/2}\sum_{k\le n}\binom{n}{k} ||g^{(k)}||_\infty ||\widehat{\xi^{n-k}f'}||_\infty
\end{align*}
Since the intersection $\{d(S, x) \le b\}\cap d(S', x) \le b\}$ is null for small enough $b$, we have 
\begin{equation}
||\phi_j^{(n)}||_2  \le C\sum_{k\le n}\binom{n}{k} ||g^{(k)}||_\infty ||\widehat{\xi^{n-k}f'}||_\infty
\end{equation}\label{phiL2bound}

Using the equation $g'(x) = (2\pi/[f])\sum_j |\phi_j(x)|^2$ (with $[f] = \lim_{x\to\infty}(f(x) - f(-x))$) we obtain
\begin{equation}\label{gbounds}
||g^{(l+1)}||_\infty \le (2\pi/[f]\sum_j\sum _k \binom {l}{k} ||\phi_j^{(l-k)}||_\infty ||\phi_j^{(k)}||_\infty.
\end{equation}
$$||g^{(l+1)}||_1\le (2\pi/[f]\sum_j\sum _k \binom {l}{k} ||\phi_j^{(l-k)}||_2||\phi_j^{(k)}||_2.$$
From (\ref{Fkernel}) we obtain 
$$f'(\xi) = (2\pi/[g])\sum_j|\widehat{\phi_j}(\xi)|^2.$$
Taking the Fourier transform we get 
$$\widehat{f'}(y-x) = (\sqrt{2\pi}/[g])\sum_j\int\phi_j(u+x)\overline{\phi_j(u+y)}du$$
After differentiating $k$ times and integrating by parts $l\le k$ times we learn that 
$$d^k/dx^k \widehat {f'}(y-x) = (-1)^l(\sqrt{2\pi}/[g]) \sum_j \int\phi_j^{(k-l)}(x+ u) \overline {\phi_j^{(l)}}(y+u)du$$ 
so that 

\begin{equation} \label{f'bound}
 ||\widehat {\xi^kf'}||_\infty \le (\sqrt{2\pi}/[g]) \sum_j ||\phi_j^{(k-l)}||_2 ||\phi_j^{(l)}||_2.
 \end{equation}

Inductively assume 

\begin{equation}\label{induction}
||\phi_j^{(l)}||_\infty \le a K^l l! \ \text{for} \ l\le n-1 \ \text{and the same estimate for} \ ||\phi_j^{(l)}||_2. 
\end{equation}

\flushleft Using the estimate (\ref{f'bound}) for $\widehat {\xi^{n-k}f'}$ and the inductive assumption (\ref{induction}) we have for $l\le n-k$, 

\begin{equation}\label{f'bound2}
||\widehat {\xi^{n-k}f'}||_\infty \le (2\pi/[g])N a^2 K^{n-k}l!(n-k-l)!.
\end{equation}
We use the $n!$ bounds 

$$\sqrt{2\pi n}(n/e)^n < n! < \sqrt{2\pi n}(n/e)^ne^{1/12n}$$.\
\flushleft in (\ref{f'bound2}). For $n-k$ even we take $l = (n-k)/2$. Otherwise take $l = (n-k -1)/2$. We obtain

\begin{equation} \label{f'bound3}
||\widehat {\xi^{n-k}f'}||_\infty \le K^{n-k}(2\pi/[g])N a^2 C_0 \sqrt{(n-k +1)}(n-k)! /2^{n-k} 
\end{equation} \

where $C_0$ is independent of $n$ and $k$. 

From (\ref{gbounds}) and (\ref{induction}) we have for $k\ge 1$

\begin{equation}\label{gbounds2}
||g^{(k)}||_\infty \le (2\pi/[f])Na^2K^{k-1}k!
\end{equation} \

We substitute the bounds (\ref{gbounds}) and (\ref{f'bound3}) into (\ref{phibound}) using the inductive assumption (\ref{induction}). We separate out the $k=0$ term below.  Let $C_1 = CC_0N(2\pi/[g])$, and $C_2 = (2\pi/[f])N C_1$. We have \
\begin{align*}
&||\phi_j^{(n)}||_\infty \le C_1a^2n!\sum_{k\le n}||g^{(k)}||_\infty\sqrt{n-k+1}(K/2)^{n-k}/k!\\
& \le C_2a^4K^{n-1}n! \sum_{1\le k\le n} \sqrt{n-k+1}/2^{n-k} + C_1a^2||g||_\infty K^n n!\sqrt{n+1}/2^n \\
&\le C_3a^4K^{n-1}n! + C_1a^2||g||_\infty K^n n!\sqrt{n+1}/2^n
\end{align*}
where $C_3 = C_2\sum_{m=0}^\infty \sqrt{m+1}/2^m$.
Choose $a$ so that $||\phi_j||_\infty \le a$, $||\phi_j||_2 \le a$.  (Note (\ref{phibound}) for $||\phi_j||_\infty$.) Choose $n_0$ so that $C_1a||g||_\infty2^{-n}\sqrt{n+1} \le 1/2$ for $n \ge n_0$.  Choose $K_0 >1 $ so that $||\phi_j^{(n)}||_\infty \le aK_0^n n!$, and $||\phi_j^{(n)}||_2 \le aK_0^nn! $ for $n\le n_0$.  Then for any $K\ge K_0$ and $K^{-1}C_3a^3 \le 1/2$, $||\phi_j^{(n)}||_\infty \le a K^n n!$ for all $n$.  Similarly for $||\phi_j^{(n)}||_2$. Thus $\phi_j$ has an analytic continuation to the strip $S_r$ with $r = K^{-1}$, bounded and in $L^2$ of any smaller strip.

\begin{lemma} \label{exponential decay}
Let $r =K^{-1}$ as above.  Then if $s<2r$, $\int e^{s|\xi|}f'(\xi) d\xi < \infty$.
\end{lemma}
\begin{proof} According to (\ref{f'bound3}) we have $|\widehat {\xi^n f'} (0)| \le C_4(K/2)^n \sqrt {n+1}n!$.  Thus  $$\int \sum_{n=0}^M ((s\xi)^{2n}/(2n)!) f'(\xi) d\xi \le C_4\sum_{n=0}^\infty (s/2r)^{2n}\sqrt{2n + 1}< \infty.$$  Thus by the monotone convergence theorem (since $f' \ge 0$) we have $\int \cosh(s\xi)f'(\xi)d\xi < \infty$.
\end{proof}
\begin{proposition}\label{Img>0}
Suppose the $\phi_j$'s are analytic in $|\text{Im}z| < r$ and $\int f'(\xi)e^{2s|\xi|} < \infty$ for $s<r$.  Then $(\text{Im}g )(\text{Im}z) \ge 0$ in this region.
\end{proposition}
\begin{proof}
From (\ref{kernel}) we have $g'(x) = (2\pi/[f])\sum_j|\phi_j(x)|^2$.  It follows that $g'$ and thus $g$ has an analytic continuation to the strip $S_r$.
We have for $x_1$ and $x_2$ real $$(1/\sqrt{2\pi})\frac{g(x_1) - g(x_2)}{x_1-x_2}\widehat{f'}(x_2 - x_1) = \sum_j \phi_j(x_1)\widebar{\phi_j(x_2)}.$$
If we analytically continue to the strip $|\text{Im}z| <r$ in both variables $x_1$ and $x_2$ we have 
$$(1/\sqrt{2\pi})\frac{g(z) - g(w)}{z-w}\widehat{f'}(w - z) = \sum_j \phi_j(z)\widebar{\phi_j}(\widebar{w}).$$
We set $z=x+iy, w = x-iy$ giving
$$(1/\sqrt{2\pi})\frac{g(x+iy) - g(x-iy)}{2iy}\widehat{f'}(-2iy) = \sum_j \phi_j(x+iy)\overline{\phi_j(x+iy)}.$$
Since $0 < \widehat{f'}(-2iy) < \infty$, $y\text{Im}g(x+iy) \ge 0$.
\end{proof}

Let us complete the proof of Theorem \ref{minimalrr'}.  We notice that because of the symmetry between $f$ and $g$ we have the existence of an $r'>0$ so that $\hat\phi_j$ and $f$ have  analytic continuations to $S_{r'}$ with $f\in K_{r'}$ and $\int g'(x)e^{2s|x|} dx <  \infty$ when $s<r'$. 

From Kato's representation theorem there is a finite positive measure $\mu_r$ such that 

$$g(x) = \int\tanh\hat r (x-t) d\mu_r(t) + c_r $$ where $\hat r = \pi/2r$.

Thus $g'(x) = \hat r\int (\cosh(\hat r(x-t))^{-2} d\mu_r(t)$ and $$\int g'(x)e^{2sx} dx = \hat r\int(\cosh(\hat rx))^{-2}e^{2sx}dx \int e^{2st} d\mu_r(t) < \infty $$ for $|s| < r'$.  This implies $r'\le \hat r$ and $\int e^{2s|t|} d\mu_r(t)$
for $s<r'$.

Of course since $g\in K_{r_1}$ if $r_1 < r$  we can also write $g(x) = \int \tanh \hat r_1(x-t) d\mu_1(t) +c_{r_1} $.  We have

$$w*\lim_{y\uparrow r} \text{Im}g(x+iy)dx = 2rd\mu_r(x)$$.  

We learn that $\int e^{2s|x|} d\mu_{r_1}< \infty$ if $s< r'$.  If we take $0<r_1 < r$ we have as before $\text{Im}g(x+ir_1) \hat f'(-2ir_1) = r_1 \sum_j|\phi_j(x+ir_1)|^2$

Thus $$\infty > \int e^{2s|x|}d\mu_{r_1}(x) = r_1\sum_j(2r_1)^{-1}\int e^{2s|x|}|\phi_j(x+ir_1)|^2 dx(\hat{f'}(-2ir_1))^{-1}$$ 

for $s< r'$.  By varying $r_1$ we learn that 

$$\int|\phi_j(x+iy)|^2 e^{2s|x|}dx < \infty ; |y| <r, s < r'.$$

 Although this is not part of Theorem \ref{minimalrr'}, using the three lines lemma with $F(z) = (\psi , e^{s(z+\cdot)}\phi_j(z + \cdot))e^{-\epsilon z^2}$ with $\psi \in C_0^\infty$ shows that the above estimate is uniform for $|y| < r-\epsilon$ and $s < r' - \epsilon$.  This completes the proof of Theorem \ref{minimalrr'}.
\end{proof}

\begin{proof}[Proof of Proposition \ref{largerrr'} ]

\begin{lemma} \label{expintegrability}
Let $t>r =K^{-1}$ and $g$ analytic in $S_t = \{z: |\text{Im} z| < t\}$ and bounded in any smaller strip.    Then $\text{Im}z\text{Im}g(z) \ge 0$ in $S_t$.
\end{lemma}

\begin{proof}
As we saw in Proposition \ref{Img>0} and Lemma \ref{exponential decay} we know that for $r = K^{-1}$,  $g$ is analytic in $S_{r}$ and $\int f'(\xi) e^{s|\xi|}d\xi < \infty $ if $s < 2r$. We have 
\begin{equation}\label{gamma_j}
    \gamma_j(x) = \frac{g(x) - g(y_j)}{x-y_j}\widehat{f'}(y_j-x).
\end{equation} 
We can analytically continue the right hand side into $S_{r_1}$ where $r_1 = \min\{t,2r\}$.  Thus $\phi_j$ is analytic in $S_{r_1}$ and is bounded in $S_{r_1 -\epsilon}$ with $\int |\phi_j(x+iy)|^2 dx \le C_\epsilon$ for all $|y| \le r_1 - \epsilon$.  It also follows that $|\phi_j(x+iy)| \rightarrow 0$ as $|x| \to \infty$ uniformly for $|y| < r_1 -\epsilon$.  Let $\hat{h}\in C_0^{\infty}(\mathbb{R})$.  By Cauchy's  theorem we have 
$$0 = \int_{-L}^L\bar{h}(x)\phi_j(x+iy)dx - \int_{-L}^L \bar{h}(x+iy)\phi_j(x)dx   $$
$$+  i\int_0^y[\bar{h}(-L+iu)\phi_j(-L -i(u-y)) -\bar{h}(L+iu)\phi_j(L -i(u-y)) ]du.$$
The integrals from $0$ to $y$ vanish in the limit $L\to \infty$ so that 
$$\int\bar{h}(x)\phi_j(x+iy)dx = \int \bar{h}(x+iy)\phi_j(x)dx   $$ or writing $\phi_{j,y}(x) = \phi_j(x+iy)$ and using the Plancherel theorem
$$\int\bar{\hat{h}}(\xi)\widehat{\phi_{j,y}}(\xi)\d\xi = \int\bar{\hat{h}}(\xi)e^{-y\xi}\hat{\phi_j}(\xi)d\xi.$$  It follows that 
$$\widehat{\phi_{j,y}}(\xi)=e^{-y\xi}\hat{\phi_j}(\xi).$$
Since this holds for $|y|< r_1$ we see that in particular $ e^{s|\xi|}\hat{\phi_j} \in L^2$ for $s < r_1$. Since 
$$f'(\xi) = (2\pi/[g])\sum_j|\widehat{\phi_j}(\xi)|^2$$ it follows that $\int f'(\xi)e^{s|\xi|} d\xi < \infty $
if $s< 2r_1$ which is an improvement over $2r$ as as long as $t > r$.)  Now we use (\ref{gamma_j}) again to show that $\phi_j$ has an analytic continuation to $S_{r_2}, r_2 = \min\{t, 2r_1\}$ which is bounded in  $S_{r_2 -\epsilon}$ with $\int |\phi_j(x+iy)|^2 dx \le C_\epsilon$ for all $|y| \le r_2 - \epsilon$.  Thus continuing we find that if $n$ is the smallest integer such that $2^n r \ge t$, we have for $m \le n, r_m = \min\{t, 2^m r\}$ and thus $r_n = t$.  We have $\phi_j(z)$ bounded in $S_{t-\epsilon}$ and $\int |\phi_j(x+iy)|^2dx \le C_\epsilon$ for $|y| < t- \epsilon$, and finally $\int f'(\xi)e^{s|\xi|}d\xi < \infty$ if $s < 2t$. The result now follows from Theorem \ref{Img>0}.
\end{proof}

The following is a corollary of the proof of Lemma \ref{expintegrability} and of Proposition \ref{Img>0}.

\begin{corollary}
Suppose $g$ and $f$ have analytic continuations to the strips $S_t$ and $S_{t'}$, bounded in any smaller strips.  Then $\int f'(\xi) e^{2s|\xi|}d\xi < \infty$ if $s< t$ and $\int g'(x) e^{2s|x|}dx < \infty$ if $s< t'$.  In addition $g\in K_t$ and $f\in K_{t'}$.  
\end{corollary}

This completes the proof of Proposition \ref{largerrr'}.
\end{proof}

\end{document}